\documentclass[12pt]{amsart}
\usepackage{amsmath,amsfonts,amssymb,amsthm,amstext,pgf,graphicx,hyperref,verbatim,lmodern,textcomp,color,young,tikz}
\usetikzlibrary{decorations}
\usepackage[mathscr]{euscript}
\usetikzlibrary{decorations.markings}
\usetikzlibrary{arrows}
\usepackage{float}
\usepackage{float}
\theoremstyle{plain}
\newtheorem{theorem}{Theorem}[section]
\newtheorem{lemma}[theorem]{Lemma}
\newtheorem{proposition}[theorem]{Proposition}

\theoremstyle{definition}
\newtheorem{defn}{Definition}[section]

\newtheorem{exmp}{Example}[section]

\title{}
\begin{document}
\title[On graphs with equal domination and total domination numbers]{On graphs with equal domination and total domination numbers}
\author[Sudip Bera]{Sudip Bera}
\address[Sudip Bera]{Faculty of Mathematics, Dhirubhai Ambani University, 
Gandhinagar, Gujarat, India}
\email{sudip\_bera@dau.ac.in}


\begin{abstract}
For a graph $G$ without isolated vertices, $\gamma(G)\le\gamma_t(G)\le 2\gamma(G)$. While graphs attaining $\gamma(G)=\gamma_t(G)$ have been studied extensively, a complete structural description in the smallest nontrivial case $\gamma(G)=2$ has remained open. We resolve this case according to girth. When $g(G)\ne 3$, we show $\gamma_t(G)=2$ forces $G$ bipartite, give an exact degree-sum criterion for this equality, and show $\gamma_t(G)\in\{2,4\}$ under the additional hypothesis $\delta(G)\ge 2$. When $g(G)=3$, we use Golumbic's vertex-multiplication operation together with known classifications of graphs of rank $2$ through $5$ to completely list the families satisfying $\gamma(G)=\gamma_t(G)=2$. As an application, we show that every graph in the extremal family of diameter-two, dominating-vertex-free graphs identified by Erd\H{o}s and R\'enyi and classified by Henning and Southey satisfies $\gamma_t(G)\in\{3,6\}$, so $\gamma=\gamma_t=2$ never occurs there. Together these results give a full structural dictionary translating $\gamma_t(G)=2$ into concrete, checkable graph-theoretic properties.
\end{abstract}

\maketitle	

\section{Introduction}

Let $G=(V,E)$ be a simple, connected, undirected graph of order $n=|V(G)|$. A set $D \subseteq V(G)$ is a \emph{dominating set} of $G$ if every vertex not in $D$ has a neighbor in $D$; the minimum cardinality of a dominating set is the \emph{domination number} $\gamma(G)$. A set $D\subseteq V(G)$ is a \emph{total dominating set} if every vertex of $G$ including those in $D$ itself has a neighbor in $D$; the minimum cardinality of such a set is the \emph{total domination number} $\gamma_t(G)$, a parameter introduced by Cockayne, Dawes, and Hedetniemi \cite{Cockayne-Dawes-Hedetniemi-1980} and since developed into a substantial branch of domination theory, surveyed comprehensively by Henning and Yeo \cite{HenningYeo-book}.

Domination and total domination are closely related but genuinely different invariants. Every total dominating set is a dominating set, and it is elementary that
\[
\gamma(G) \;\le\; \gamma_t(G) \;\le\; 2\gamma(G)
\]
for every graph $G$ without isolated vertices. Both inequalities can be tight, and a natural and recurring theme in the literature is to characterize the graphs for which the \emph{lower} bound is attained, that is, for which
$\gamma(G) = \gamma_t(G).$
Such graphs are of interest for at least two reasons. First, from a purely structural standpoint, equality forces every minimum dominating set to already be ``self-sufficient'' under the stronger total domination requirement, which imposes rigid adjacency conditions on the dominating set itself. Second, from an algorithmic standpoint, graphs satisfying $\gamma=\gamma_t$ are precisely those for which the (generally harder to compute) total domination number can be read off immediately from an ordinary minimum dominating set, which is of practical value in network design and facility location models where total domination corresponds to guaranteeing that every dominator itself receives ``backup'' coverage from another dominator.

The extremal behaviour of graphs with small diameter has played a prominent role in this circle of ideas. Goddard and Henning \cite{GoddardHenning2002} showed that every planar graph of diameter $2$ satisfies $\gamma_t(G) \le 3$; MacGillivray and Seyffarth \cite{MacGillivraySeyffarth1996} bounded the domination number of planar graphs of diameter $3$; and Dorfling, Goddard, and Henning \cite{DorflingGoddardHenning2006} refined the diameter-$3$ planar bound for total domination according to the radius of $G$. These results highlight that even mild restrictions on diameter can force strong upper bounds on $\gamma_t(G)$, but they say little about exactly when equality $\gamma=\gamma_t$ occurs.

The present paper focuses on the smallest interesting instance of this problem: graphs with $\gamma(G)=2$. This case is already far from trivial. Since $\gamma_t(G)\ge \gamma(G)=2$ always, the only question is whether $\gamma_t(G)$ equals $2$ or exceeds it. 
We investigate the two natural regimes distinguished by the girth $g(G)$ of the graph $G$.

\begin{itemize}
\item When $g(G) \ne 3$ that is $G$ has no triangles supporting the dominating pair we prove that $\gamma_t(G)=2$ forces $G$ to be bipartite, and we characterize exactly which bipartite graphs with $\gamma(G)=2$ achieve $\gamma_t(G)=2$ via an explicit degree-sum condition on a dominating pair with disjoint neighborhoods. As a further refinement, when the minimum degree $\delta(G) \ge 2$ we show the total domination number is confined to the two values $\{2,4\}$, ruling out $\gamma_t(G)=3$ entirely by an odd-cycle parity argument.

\item When $g(G)=3$, triangles are available and the classification becomes considerably richer. Here we exploit the graph-multiplication operation of Golumbic \cite{Graph-multiplication-Golumbic} which replaces each vertex of a base graph by an independent ``blow-up'' set together with existing classifications of all graphs of matrix rank $2, 3, 4,$ and $5$ due to \cite{Rank-4-graph-LAA,LAA-rank-5-graph-charecterization,Nulity-of-graphs-ELA}. We show that both girth and total domination number behave predictably under this operation, which allows the rank classification to be lifted verbatim into a complete list of graph families realizing $\gamma(G)=\gamma_t(G)=2$ with $g(G)=3$.
\end{itemize}

Finally, we connect our results to a classical extremal problem. Erd\H{o}s and R\'{e}nyi \cite{ERDOS-RENYEL} determined the minimum number of edges in a diameter-$2$ graph with no dominating vertex, and Henning and Southey \cite{Henning-Michael} later characterized the extremal family $\mathcal{H}$ attaining this minimum. We show that every graph in $\mathcal{H}$ satisfies $\gamma_t(\Gamma) \in \{3,6\}$, so equality $\gamma=\gamma_t$ never arises in this extremal setting a fact that was not previously observed and which illustrates the tension between edge-minimality (forced by diameter $2$) and the adjacency requirements needed for total domination equality.

Taken together, our results provide, to the best of our knowledge, the first complete structural dictionary translating the numerical condition $\gamma(G)=\gamma_t(G)=2$ into explicit, verifiable graph properties, spanning bipartite and triangle-containing graphs alike, and tying the classification to both classical extremal domination theory and the algebraic theory of graphs of small rank.

\subsection{Terminology and Notation}

Throughout the paper, $G$ denotes a simple connected graph on $n$ vertices. The \emph{order} and \emph{size} of a graph are the cardinalities of its vertex set and edge set, respectively. We write $P_n$, $C_n$, and $K_n$ for the path, cycle, and complete graph on $n$ vertices, and $\deg(v)$, $\delta(G)$, $N(v), r(G)$ for the degree of a vertex $v$, the minimum degree of $G$, the neighborhood of $v$, and the rank of $G$ respectively. For $u,v\in V(G)$, $d(u,v)$ denotes their distance, and $\operatorname{diam}(G)$ denotes the diameter of $G$. We denote the set $\{1, \cdots, n\}$ by $[n].$

We now recall a graph operation, due to Golumbic \cite{Graph-multiplication-Golumbic}, that serves as our principal tool for building infinite families of graphs from a finite list of reduced graphs.

\begin{defn}\label{def:mult-of-vertices}
Let $G$ be a simple connected graph with vertex set $V(G)=\{v_1,\ldots,v_n\}$, and let $m=(m_1,\ldots,m_n)$ be a vector of positive integers. The \emph{blow-up} $G\bigodot m$ is obtained from $G$ by replacing each vertex $v_i$ with an independent set
\[
V_i = \{v_i^1,\ldots,v_i^{m_i}\}
\]
of $m_i$ vertices, and joining $v_i^s$ to $v_j^t$ if and only if $v_iv_j\in E(G)$. We call $V_i$ the \emph{part of $G\bigodot m$ corresponding to $v_i$}, and say that $G\bigodot m$ is obtained from $G$ by \emph{multiplication of vertices} (Golumbic \cite{Graph-multiplication-Golumbic}, p.~53).

If $m_j=1$ for all $j\ne i$, we say $G\bigodot m$ is obtained from $G$ by \emph{multiplying the single vertex $v_i$}; if moreover $m_i=2$, this is called \emph{duplicating} $v_i$ (Henning and Southey \cite{Henning-Michael}). If $G$ has a vertex of degree $2$, \emph{degree-$2$ vertex duplication} denotes the operation of duplicating any such vertex.

For graphs $G_1,\ldots,G_k$, we write $\mathcal{M}(G_1,\ldots,G_k)$ for the family of all graphs obtainable from some $G_\ell$ by multiplication of vertices. Since $G\bigodot(1,\ldots,1)=G$, each $G_\ell$ itself belongs to $\mathcal{M}(G_1,\ldots,G_k)$.
\end{defn}

Throughout the remainder of the paper, $m=(m_1,\ldots,m_n)$ always denotes a vector of positive integers, and $G\bigodot m$ the corresponding blow-up of $G$ as in Definition~\ref{def:mult-of-vertices}.

\section{Triangle-Free Graphs: A Bipartite Reduction}

We now consider the case $g(G)\neq 3$, that is, triangle-free graphs. Here the absence of $3$-cycles imposes a surprisingly strong constraint: any witness of $\gamma_t(G)=2$ forces the underlying graph to be bipartite.

\begin{proposition}\label{prop:Td=2-bipartite}
Let $G$ be a connected graph with $\gamma(G)=2$ and $g(G)\neq 3$. If $\gamma_t(G)=2,$ then $G$ is bipartite.
\end{proposition}

\begin{proof}
Let $D=\{u,v\}$ be a total dominating set of $G$. 
Denote
\[
N(u)=\{u_1,\dots,u_k\}
\quad \text{and} \quad
N(v)=\{v_1,\dots,v_r\}.
\]
Since $D$ is minimum dominating set $N(u)$ and $N(v)$ are non-empty.
We prove the following two claims.

\medskip
\noindent
\textit{Claim I.} $N(u)$ and $N(v)$ are independent sets.

Suppose that $u_i\sim u_j$ for some $i\neq j$. Then
$u_i \sim u \sim u_j \sim u_i$ forms a cycle of length $3$, contradicting the assumption that $g(G)\neq 3$. Hence, $N(u)$ is an independent set. By a similar argument, $N(v)$ is also an independent set.

\medskip
\noindent
\textit{Claim II.} $N(u)\cap N(v)=\emptyset$.

Suppose there exists a vertex $x\in N(u)\cap N(v)$. Then $x\sim u$ and $x\sim v$. Since $u\sim v$, it follows that
$u \sim x \sim v \sim u$
is a cycle of length $3$, again contradicting $g(G)\ge 4$. Therefore, $N(u)\cap N(v)=\emptyset$. Now set
\[
A=\{u\}\cup N(v)
\quad \text{and} \quad
B=\{v\}\cup N(u).
\]
By Claims I and II, both $A$ and $B$ are independent sets. Moreover, since $D=\{u,v\}$ is a dominating set, every vertex of $G$ is adjacent to either $u$ or $v$, and hence $V(G)=A\cup B.$ Thus, $G$ admits a bipartition $(A,B)$, and therefore $G$ is bipartite.
\end{proof}

\begin{lemma}\label{Lemma:D-td-iff-deg=V(G)}
Let $G$ be a connected graph with $g(G)\neq 3$, $\gamma(G)=2.$ 
Let $D=\{u,v\}$ be a dominating set of $G$. Then $D$ is a total dominating set if and only if the following hold:
\begin{itemize}
\item $N(u)\cap N(v)=\emptyset,$ and
\item $\deg(u)+\deg(v)=|V(G)|.$
\end{itemize}
\end{lemma}

\begin{proof}
Suppose that $D=\{u,v\}$ is a total dominating set of $G$. Since $g(G)\neq 3$, it follows by the same argument used in the proof of the claims in Proposition~\ref{prop:Td=2-bipartite} that
\[
N(u)\cap N(v)=\emptyset,
\]
and moreover, each of $N(u)$ and $N(v)$ is an independent set. Because $D$ is a total dominating set, every vertex of $G$ lies in $N(u)\cup N(v)$. Hence,
\[
|V(G)| \le \deg(u) + \deg(v).
\]
Moreover, since $N(u)\cap N(v)=\emptyset$, we have
\[
|N(u)\cup N(v)| = \deg(u)+\deg(v).
\]

Now, $D$ is a total dominating set so, $u\in N(v)$ and $v\in N(u)$, and consequently
\[
V(G)=N(u)\cup N(v),
\]
which implies
\[
\deg(u)+\deg(v)=|V(G)|.
\]

Conversely, if $N(u)\cap N(v)=\emptyset,$ and $\deg(u)+\deg(v)=|V(G)|$, then every vertex of $G$ belongs to $N(u)\cup N(v)$. In particular, $u\in N(v)$ and $v\in N(u)$, and hence $u\sim v$. Therefore, $D$ is a total dominating set. This completes the proof.
\end{proof}

\begin{theorem}
Let $G$ be a connected graph with $g(G)\neq 3$, and $\gamma(G)=2.$ 
Then $\gamma_t(G)=2$ if and only if the following conditions are satisfied:
\begin{itemize}
    \item $G$ is bipartite, and
    \item $G$ admits a dominating set $D=\{u,v\}$ such that 
    \[
    N(u)\cap N(v)=\emptyset 
    \quad \text{and} \quad 
    \deg(u)+\deg(v)=|V(G)|.
    \]
\end{itemize}
\end{theorem}

\begin{proof}
First, suppose that $G$ is bipartite and that $G$ has a dominating set 
$D=\{u,v\}$ satisfying the conditions, \[
    N(u)\cap N(v)=\emptyset 
    \quad \text{and} \quad 
    \deg(u)+\deg(v)=|V(G)|.
    \]
Then, by Lemma~\ref{Lemma:D-td-iff-deg=V(G)}, $D$ is a total dominating set. 
Since $\gamma(G)=2$, it follows that $\gamma_t(G)=2$.

Conversely, suppose that $\gamma_t(G)=2$ and let $D=\{u,v\}$ be a total dominating set of $G$. 
Then $u\sim v$. By Proposition~\ref{prop:Td=2-bipartite}, $G$ is bipartite with bipartition
\[
A=N(u)\cup\{v\}
\quad \text{and} \quad
B=N(v)\cup\{u\}.
\] That is, $N(u)\cap N(v)=\emptyset.$
Moreover, since $D$ is a total dominating set, we have $u$ and $v$ are edge connected and hence
\[
\deg(u)+\deg(v)=|V(G)|.
\]

This completes the proof.
\end{proof}

\begin{theorem}
Let $G$ be a connected bipartite graph with $g(G)\neq 3$, $\gamma(G)=2$, and $\delta(G)\ge 2$. Then $\gamma_t(G)\in\{2,4\}$.
\end{theorem}

\begin{proof}
Let $D=\{u,v\}$ be a minimum dominating set of $G$. Denote
\[
N(u)=\{u_1,\dots,u_k\}
\quad \text{and} \quad
N(v)=\{v_1,\dots,v_r\}.
\]

Since $D$ is a \emph{minimum} dominating set, there exist vertices 
$u_i \in N(u)$ and $v_j \in N(v)$ such that $u_i \nsim v$ and 
$v_j \nsim u$; otherwise, one of the vertices $u$ or $v$ could be removed while preserving domination, contradicting the minimality of $D$.

By Claim~I of Proposition~\ref{prop:Td=2-bipartite}, the sets $N(u)$ and $N(v)$ are independent. Moreover, since $\delta(G)\ge 2$, every vertex has degree at least two. Hence, there exist vertices $u_s \in N(u)$ and $v_t \in N(v)$ such that
$u_i \sim v_t
\text{ and } 
u_s \sim v_j.$
Suppose, for a contradiction, that $\gamma_t(G)=3$. Then there exists a vertex $x$ such that 
$u \sim x 
 \text{ and }
v \sim x.$
Consequently, we obtain a cycle
$C:\; u \sim x \sim v \sim v_t \sim u_i \sim u$
of length $5$. 
This contradicts the fact that $G$ is bipartite, since bipartite graphs contain no odd cycles. Therefore, $\gamma_t(G)\neq 3$. This completes the proof.
\end{proof}

\section{Girth Three: A Finite Classification via Rank}

Having settled the triangle-free case in the previous section, we now turn to the complementary and considerably more delicate setting in which $G$ contains a triangle, that is, $g(G)=3$. Here our approach exploits the adjacency-matrix rank $r(G)$ as a classifying invariant: since graphs of small rank have been completely characterized, up to the vertex-multiplication (``blow-up'') operation of Definition~\ref{def:mult-of-vertices}, in terms of a short list of reduced graphs, the problem of determining when $\gamma(G)=\gamma_t(G)=2$ within each rank class reduces to a finite check on these reduced graphs.

\begin{lemma}\label{lemma:girth-blowup}
$g(G)=3$ if and only if $g(G\bigodot m)=3$.
\end{lemma}

\begin{proof}
Suppose first that $g(G)=3$. Then $G$ contains a triangle, say with vertex set ${v_i,v_j,v_k}$. By the definition of $G\bigodot m$, the vertices $v_i,v_j,v_k$ are replaced by nonempty independent sets $V_i,V_j,V_k$, respectively, and every vertex of $V_i$ is adjacent to every vertex of $V_j$ and $V_k$, while every vertex of $V_j$ is adjacent to every vertex of $V_k$. Choosing arbitrary vertices $x_i\in V_i$, $x_j\in V_j$, and $x_k\in V_k$, we obtain a triangle in $G\bigodot m$. Hence $g(G\bigodot m)=3$.

Conversely, suppose that $g(G\bigodot m)=3$. Then $G\bigodot m$ contains a triangle, say with vertices $x_i,x_j,x_k$. Since each set $V_\ell$ corresponding to a vertex $v_\ell$ of $G$ is independent, no two vertices of the triangle can belong to the same part. Thus $x_i,x_j,x_k$ belong to three distinct parts $V_i,V_j,V_k$. By the definition of $G\bigodot m$, the adjacency of $x_i$ and $x_j$, $x_j$ and $x_k$, and $x_k$ and $x_i$ implies that $v_i v_j$, $v_jv_k$, and $v_kv_i$ are edges of $G$. Consequently, $v_i,v_j,v_k$ form a triangle in $G$, and therefore $g(G)=3$.
\end{proof}

\begin{lemma}\label{Lemma;Dom(GOm)geDom(G)}
$\gamma(G\bigodot m)\ \geq\ \gamma(G)$.
\end{lemma}

\begin{proof}
Let $D=\{v_{i_1}^{j_1},\ldots,v_{i_t}^{j_t}\}$ be a minimum dominating set of $G\bigodot m$, so $t=\gamma(G\bigodot m)$, and set $D_1=\{v_{i_1},\ldots,v_{i_t}\}\subseteq V(G)$. For any $v_\ell\notin D_1$, some vertex of $V_\ell$ is dominated by $D$, hence adjacent to some $v_{i_s}^{j_s}\in D$; by construction this forces $v_\ell v_{i_s}\in E(G)$. Thus $D_1$ dominates $G$, so $\gamma(G)\leq |D_1|\leq t=\gamma(G\bigodot m)$.
\end{proof}

\begin{lemma}\label{Lemma:TD(GammaI)EqualToTD(GammaOm)}
$\gamma_t(G\bigodot m) = \gamma_t(G)$.
\end{lemma}

\begin{proof}
Let $D$ be a minimum total dominating set of $G$, and set $D'=\{v_i^1 : v_i\in D\}$. For any vertex $v_j^s$ of $G\bigodot m$, the vertex $v_j$ has a neighbor $v_i\in D$ (as $D$ totally dominates $G$), and by construction $v_j^s\sim v_i^1$. Hence $D'$ totally dominates $G\bigodot m$, giving $\gamma_t(G\bigodot m)\leq |D'|=\gamma_t(G)$.

Let $D$ be a minimum total dominating set of $G\bigodot m$, and set $D_1=\{v_i : V_i\cap D\neq\emptyset\}\subseteq V(G)$. Since each $V_i$ is independent, every vertex of $D$ has its totally-dominating neighbor in some distinct part $V_j$; thus for each $v_i\in D_1$ there is $v_j\in D_1$ with $v_iv_j\in E(G)$, and moreover every $v_i\in V(G)$ satisfies this same property by the domination of $G\bigodot m$. Hence $D_1$ is a total dominating set of $G$, giving $\gamma_t(G)\leq |D_1|\leq |D|=\gamma_t(G\bigodot m)$.
Combining both inequalities completes the proof.
\end{proof}
\begin{lemma}\label{lemma:dn=tdn-implydn=dn-inmultiplication}
If $\gamma(G)=\gamma_t(G)$, then $\gamma(G)=\gamma(G\bigodot m)$.
\end{lemma}

\begin{proof}
Let $D$ be a minimum total dominating set of $G$. Since $D$ is a total
dominating set, for every $u\in D$, there exists $v\in D$ such that
$u\sim v$. By the construction of $G\bigodot m$, it follows that $D$
dominates $G\bigodot m$ for every $m=(m_1,\ldots,m_n)$. Hence
$\gamma(G\bigodot m)\leq |D|=\gamma_t(G)=\gamma(G).$
On the other hand, since $G$ is an induced subgraph of
$G\bigodot m$, every dominating set of $G\bigodot m$ induces a
dominating set of $G$. Therefore
$\gamma(G)\leq\gamma(G\bigodot m).$
Consequently,
$\gamma(G)=\gamma(G\bigodot m),$
as required.
\end{proof}

\begin{lemma}\label{lemma:dn=1-tdn=dn=2-iff}
Suppose, in addition, that $v_i$ is a dominating vertex of $G$. Then
$\gamma(G\bigodot m)=\gamma_t(G\bigodot m)=2$ if and only if $m_i\geq 2$.
\end{lemma}

\begin{proof}
Suppose first that $m_i\geq 2$. Since $v_i$ is a dominating vertex of $G$, it follows that $\gamma(G\bigodot m)=2$. Moreover, by Lemma~\ref{Lemma:TD(GammaI)EqualToTD(GammaOm)}, we have $\gamma_t(G\bigodot m)=2$.

Conversely, suppose that $\gamma(G\bigodot m)=\gamma_t(G\bigodot m)=2$. If $m_i=1$, then $G\bigodot m=G$, and hence $v_i$ is a dominating vertex of $G\bigodot m$. This implies that $\gamma(G\bigodot m)=1$, contradicting the assumption that $\gamma(G\bigodot m)=2$. Therefore, $m_i\geq 2$.
\end{proof}

\begin{lemma}[\cite{Nulity-of-graphs-ELA,Nulity-bicyclic-graph,Rank-of-graphs-Sc}]\label{lemma: cha of graphs rank leq 3}
	Let $r(G)$ be the rank of $G.$ Then 
	\begin{enumerate}
		\item $r(G)=2$ if and only if $G\in \mathcal{M}(K_2)$
		\item $r(G)=3$ if and only if $G\in \mathcal{M}(K_3).$
	\end{enumerate}	
\end{lemma}

\begin{lemma}\label{lemma:rank-2-dn=tdn=2-iff}
Let $G\in\mathcal{M}(K_2)$. Then
$\gamma(G)=\gamma_t(G)=2$
if and only if
$G=K_2\bigodot m,$
where $m=(m_1,m_2)$ with $m_1,m_2\geq 2$.
\end{lemma}

\begin{proof}
Suppose first that
$G=K_2\bigodot m,
 m=(m_1,m_2).$
If $m_i=1$ for some $i\in\{1,2\}$, then $G$ has a dominating
vertex, and hence
$\gamma(G)=1.$
On the other hand, by Lemma~\ref{Lemma:TD(GammaI)EqualToTD(GammaOm)},
every graph in $\mathcal{M}(K_2)$ has total domination number $2.$ So, in this case
$\gamma(G)\neq\gamma_t(G).$

Now suppose that $m_1, m_2\geq2$. Then $G$ is a complete bipartite
graph with at least two vertices in each part. In particular, $G$ is
not complete, and hence
$\gamma(G)=2.$
Moreover, choosing one vertex from each part gives a total dominating
set of cardinality $2$. Consequently,
$\gamma_t(G)=2.$
Therefore,
$\gamma(G)=\gamma_t(G)=2.$

Conversely, let $G\in\mathcal{M}(K_2)$ and suppose that
$\gamma(G)=\gamma_t(G)=2.$
If $m_i=1$ for some $i\in\{1,2\}$, then $G$ has a dominating vertex,
which would imply $\gamma(G)=1$, a contradiction. Hence
$m_1,m_2\geq2$. Therefore,
$G=K_2\bigodot(m_1,m_2),
m_1,m_2\geq2.$
This completes the proof.
\end{proof}

\begin{lemma}\label{lemma:rank-3-dn=tdn=2-iff}
Let $G\in\mathcal{M}(K_3)$. Then $\gamma(G)=\gamma_t(G)=2$ if and only if
$G=K_3\bigodot m,$
where $m=(m_1,m_2,m_3)$ and $\max\{m_1,m_2,m_3\}\geq2$.
\end{lemma}

\begin{proof}
Suppose first that $G=K_3\bigodot m$, where $m=(m_1,m_2,m_3)$.

If $m_1=m_2=m_3=1$, then $G\cong K_3$. Consequently,
$\gamma(G)=1$ and $\gamma_t(G)=2$, and hence
$\gamma(G)\neq\gamma_t(G)$.

Now assume that $\max\{m_1,m_2,m_3\}\geq2$. In this case, $G$ is not complete, and hence $\gamma(G)=2$. Moreover, by Lemma~\ref{Lemma:TD(GammaI)EqualToTD(GammaOm)}, the graph $G$ admits a total dominating set of cardinality $2$. Thus $\gamma_t(G)=2$, and consequently
$\gamma(G)=\gamma_t(G)=2$.

Conversely, suppose that $G\in\mathcal{M}(K_3)$ and
$\gamma(G)=\gamma_t(G)=2$. If $m_1=m_2=m_3=1$, then
$G\cong K_3$, which implies $\gamma(G)=1$, a contradiction.
Therefore, $\max\{m_1,m_2,m_3\}\geq2$. Hence
$G=K_3\bigodot m$, where $m=(m_1,m_2,m_3)$ satisfies
$\max\{m_1,m_2,m_3\}\geq2$.
\end{proof}

\begin{proposition}\label{prop:rank-2,3-d-td=2}
Let $G$ be a graph with $g(G)=3$ and $r(G)\in\{2,3\}$. Then
$\gamma(G)=\gamma_t(G)=2$ if and only if
$G=K_3\bigodot m$, where $m=(m_1,m_2,m_3)$ with $\max\{m_1,m_2,m_3\}\geq2$.
\end{proposition}

\begin{proof}
By Lemma~\ref{lemma: cha of graphs rank leq 3}, the graphs of rank
$2$ and $3$ are characterized by the families described therein. In
particular, every graph in $\mathcal{M}(K_2)$ is complete bipartite
and hence has girth $4$. Thus, no graph in $\mathcal{M}(K_2)$ satisfies
the condition $g(G)=3$.

It remains to consider graphs in $\mathcal{M}(K_3)$. The desired
characterization then follows immediately from
Lemma~\ref{lemma:rank-3-dn=tdn=2-iff}.
\end{proof}

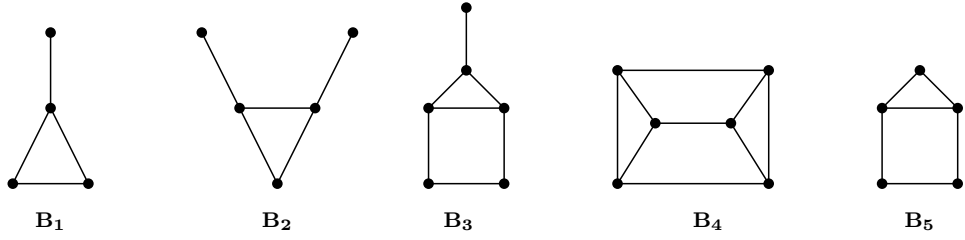
\begin{figure}
		\tiny
		\tikzstyle{ver}=[]
		\tikzstyle{vert}=[circle, draw, fill=black!100, inner sep=0pt, minimum width=4pt]
		\tikzstyle{vertex}=[circle, draw, fill=black!00, inner sep=0pt, minimum width=4pt]
		\tikzstyle{edge} = [draw,thick,-]
		\tikzstyle{node_style} = [circle,draw=blue,fill=blue!20!,font=\sffamily\Large\bfseries]
		\centering
		\begin{tikzpicture}[scale=1]
		\tikzstyle{edge_style} = [draw=black, line width=2mm, ]
		\tikzstyle{node_style} = [draw=blue,fill=blue!00!,font=\sffamily\Large\bfseries]
		%
		%
		%
		\draw[line width=.2 mm] (0,-4) -- (1,-4);
		\draw[line width=.2 mm] (1,-4) -- (.5,-3);
		\draw[line width=.2 mm] (0,-4) -- (.5,-3);
		\draw[line width=.2 mm] (.5,-2) -- (.5,-3);
		\fill[black!100!](0,-4) circle (.07);
		\fill[black!100!](1,-4) circle (.07);
		\fill[black!100!](.5,-3) circle (.07);
		\fill[black!100!](.5,-2) circle (.07);
		\node (A4) at (0.5,-4.5)  {$\bf{B_1}$};
		\draw[line width=.2 mm] (2.5,-2) -- (3,-3);
		\draw[line width=.2 mm] (3.5,-4) -- (3,-3);
		\draw[line width=.2 mm] (3.5,-4) -- (4,-3);
		\draw[line width=.2 mm] (4,-3) -- (4.5,-2);
		\draw[line width=.2 mm] (3,-3) -- (4,-3);
		\fill[black!100!](2.5,-2) circle (.07);
		\fill[black!100!](3.5,-4) circle (.07);
		\fill[black!100!](3,-3) circle (.07);
		\fill[black!100!](4,-3) circle (.07);
		\fill[black!100!](4.5,-2) circle (.07);
		\node (A4) at (3.5,-4.5)  {$\bf{B_2}$};
		\draw[line width=.2 mm] (6,-2.5) -- (6,-1.67);
		\draw[line width=.2 mm] (6,-2.5) -- (5.5,-3);
		\draw[line width=.2 mm] (6,-2.5) -- (6.5,-3);
		\draw[line width=.2 mm] (5.5,-4) -- (6.5,-4);
		\draw[line width=.2 mm] (5.5,-4) -- (5.5,-3);
		\draw[line width=.2 mm] (6.5,-4) -- (6.5,-3);
		\draw[line width=.2 mm] (5.5,-3) -- (6.5,-3);
		\fill[black!100!](6,-1.67) circle (.07);
		\fill[black!100!](6,-2.5) circle (.07);
		\fill[black!100!](5.5,-4) circle (.07);
		\fill[black!100!](5.5,-3) circle (.07);
		\fill[black!100!](6.5,-3) circle (.07);
		\fill[black!100!](6.5,-4) circle (.07);
		\node (A4) at (5.9,-4.5)  {$\bf{B_3}$};
        \begin{scope}[shift={(2.5,4)}]
		\draw[line width=.2 mm] (9.5,-6.5) -- (9,-7);
		\draw[line width=.2 mm] (9.5,-6.5) -- (10,-7);
		\draw[line width=.2 mm] (9,-8) -- (10,-8);
		\draw[line width=.2 mm] (9,-8) -- (9,-7);
		\draw[line width=.2 mm] (10,-8) -- (10,-7);
		\draw[line width=.2 mm] (9,-7) -- (10,-7);
		\fill[black!100!](9.5,-6.5) circle (.07);
		\fill[black!100!](9,-8) circle (.07);
		\fill[black!100!](9,-7) circle (.07);
		\fill[black!100!](10,-7) circle (.07);
		\fill[black!100!](10,-8) circle (.07);
		\node (A4) at (9.5,-8.5)  {$\bf{B_5}$};
        \end{scope}
        \begin{scope}[shift={(3,4)}]
		\draw[line width=.2 mm] (5,-8) -- (7,-8);
		\draw[line width=.2 mm] (5,-8) -- (5,-6.5);
		\draw[line width=.2 mm] (7,-8) -- (7,-6.5);
		\draw[line width=.2 mm] (5,-6.5) -- (7,-6.5);
		\draw[line width=.2 mm] (5.5,-7.2) -- (5,-8);
		\draw[line width=.2 mm] (5.5,-7.2) -- (5,-6.5);
		\draw[line width=.2 mm] (5.5,-7.2) -- (6.5,-7.2);
		\draw[line width=.2 mm] (7,-8) -- (6.5,-7.2);
		\draw[line width=.2 mm] (7,-6.5) -- (6.5,-7.2);
		\fill[black!100!](5.5,-7.2) circle (.07);
		\fill[black!100!](6.5,-7.2) circle (.07);
		\fill[black!100!](5,-8) circle (.07);
		\fill[black!100!](7,-8) circle (.07);
		\fill[black!100!](7,-6.5) circle (.07);
		\fill[black!100!](5,-6.5) circle (.07);
		\node (A4) at (6.2,-8.5)  {$\bf{B_4}$};
        \end{scope}
		\end{tikzpicture}
		\caption{Some reduced graphs of rank $4$}
		\label{fig:rank 4 graphs}	
	\end{figure}

\begin{lemma}[\cite{Rank-4-graph-LAA}]\label{lemma: cha of graphs of rank 4}
The rank $r(G)=4$ if and only if $G\in \mathcal{M}(B_1, B_2, B_3, B_4, B_5, K_4, P_4, P_5),$ where the graphs $B_1, B_2, B_3, B_4, B_5$ are depicted in Figure \ref{fig:rank 4 graphs}.
\end{lemma}

\begin{proposition}\label{prop:rank-4-d-td=2}
Let $G$ be a graph with $g(G)=3$ and $r(G)=4$. Then $\gamma(G)=\gamma_t(G)=2$ if and only if $G$ belongs to one of the following classes:
\begin{enumerate}
\item $G\in\mathcal{M}\bigl(B_2,B_4,B_5\bigr)$; or
\item $G\in\mathcal{M}\bigl(B_1,K_4\bigr)$, with $m_i\geq 2$ for at least one index $i$ corresponding to a dominating vertex $v_i$ of the base graph.
\end{enumerate}
\end{proposition}

\begin{proof}
Suppose first that $G\in\mathcal{M}\bigl(B_2,B_4,B_5\bigr)$. Since $\gamma(B_j)=\gamma_t(B_j)=2$ for $j\in{2,4,5}$, Lemmas~\ref{Lemma:TD(GammaI)EqualToTD(GammaOm)} and~\ref{lemma:dn=tdn-implydn=dn-inmultiplication} imply that $\gamma(G)=\gamma_t(G)=2$ for every vector $m$ of positive integers.

Next, suppose that $G\in\mathcal{M}(B_1,K_4)$, where $m=(m_1,m_2,m_3,m_4)$ with $m_i\geq 1$ for $1\leq i\leq 4$ and $m_i\geq 2$ for at least one $i,$ with $v_i$ denoting the corresponding dominating vertex. Since both the graphs $B_1$ and $K_4$ have dominating vertices, so Lemma~\ref{lemma:dn=1-tdn=dn=2-iff} yields $\gamma(G)=\gamma_t(G)=2$.

Conversely, let $G$ be a graph with $g(G)=3$ and $r(G)=4$ such that $\gamma(G)=\gamma_t(G)=2$. By Lemma~\ref{lemma: cha of graphs of rank 4}, we have $G\in\mathcal{M}\bigl(B_1,B_2,B_3,B_4,B_5,K_4\bigr)$. However, $G$ cannot belong to $\mathcal{M}(B_3)$, since $\gamma(B_3)\geq 3$ and, by Lemma~\ref{Lemma;Dom(GOm)geDom(G)}, we have $\gamma(B_3\bigodot m)\geq 3$ for every vector $m$ of positive integers.

It remains to consider the class $\mathcal{M}(B_1,K_4)$. If $m_i=1$ for every dominating vertex $v_i$, then clearly $\gamma(G)=1$, contrary to the assumption that $\gamma(G)=2$. Therefore, $m_i\geq 2$ for at least one $i$. Thus, $G$ belongs to one of the two stated classes, completing the proof.
\end{proof}

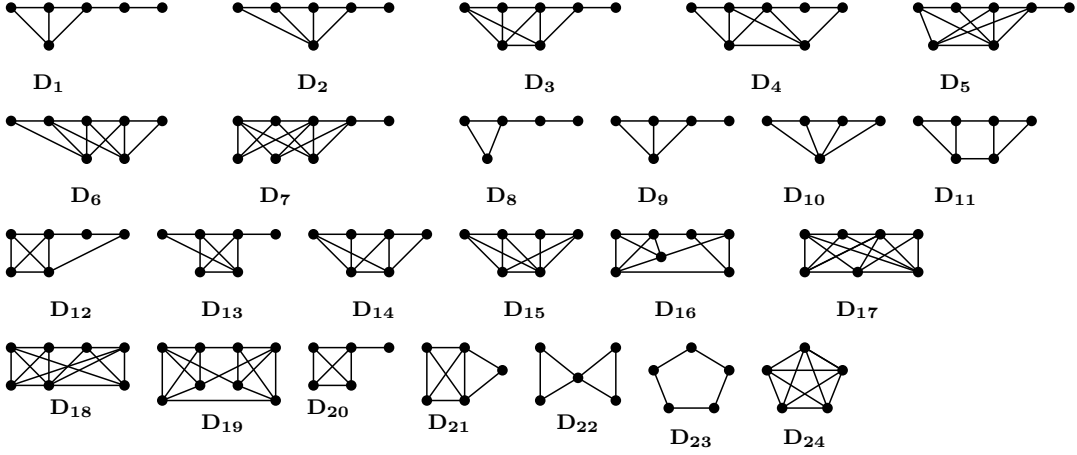
\begin{figure}
		\tiny
		\tikzstyle{ver}=[]
		\tikzstyle{vert}=[circle, draw, fill=black!100, inner sep=0pt, minimum width=4pt]
		\tikzstyle{vertex}=[circle, draw, fill=black!00, inner sep=0pt, minimum width=4pt]
		\tikzstyle{edge} = [draw,thick,-]
		\tikzstyle{node_style} = [circle,draw=blue,fill=blue!20!,font=\sffamily\Large\bfseries]
		\centering
		\begin{tikzpicture}[scale=1]
		\tikzstyle{edge_style} = [draw=black, line width=2mm, ]
		\tikzstyle{node_style} = [draw=blue,fill=blue!00!,font=\sffamily\Large\bfseries]
		\fill[black!100!](0,0) circle (.07);
		\fill[black!100!](.5,0) circle (.07);
		\fill[black!100!](1,0) circle (.07);
		\fill[black!100!](1.5,0) circle (.07);
		\fill[black!100!](2,0) circle (.07);
		\fill[black!100!](.5,-.5) circle (.07);
		\draw[line width=.2 mm] (0,0) -- (.5,0);
		\draw[line width=.2 mm] (.5,0) -- (1.5,0);
		\draw[line width=.2 mm] (2,0) -- (1.5,0);
		\draw[line width=.2 mm] (.5,-.5) -- (0,0);
		\draw[line width=.2 mm] (.5,-.5) -- (.5,0);
		\draw[line width=.2 mm] (.5,-.5) -- (1,0);
		\node (A4) at (.5,-1)  {$\bf{D_1}$};
		\fill[black!100!](3,0) circle (.07);
		\fill[black!100!](3.5,0) circle (.07);
		\fill[black!100!](4,0) circle (.07);
		\fill[black!100!](4.5,0) circle (.07);
		\fill[black!100!](5,0) circle (.07);
		\fill[black!100!](4,-.5) circle (.07);
		\draw[line width=.2 mm] (3,0) -- (3.5,0);
		\draw[line width=.2 mm] (3.5,0) -- (4,0);
		\draw[line width=.2 mm] (4,0) -- (4.5,0);
		\draw[line width=.2 mm] (4.5,0) -- (5,0);
		\draw[line width=.2 mm] (4,-.5) -- (3,0);
		\draw[line width=.2 mm] (4,-.5) -- (3.5,0);
		\draw[line width=.2 mm] (4,-.5) -- (4,0);
		\draw[line width=.2 mm] (4,-.5) -- (4.5,0);
		\node (A4) at (4,-1)  {$\bf{D_2}$};
		\fill[black!100!](6,0) circle (.07);
		\fill[black!100!](6.5,0) circle (.07);
		\fill[black!100!](7,0) circle (.07);
		\fill[black!100!](7.5,0) circle (.07);
		\fill[black!100!](8,0) circle (.07);
		\fill[black!100!](6.5,-.5) circle (.07);
		\fill[black!100!](7,-.5) circle (.07);
		\draw[line width=.2 mm] (6.5,-.5) -- (7,-.5);
		\draw[line width=.2 mm] (6,0) -- (6.5,0);
		\draw[line width=.2 mm] (6.5,0) -- (7,0);
		\draw[line width=.2 mm] (7,0) -- (7.5,0);
		\draw[line width=.2 mm] (7.5,0) -- (8,0);
		\draw[line width=.2 mm] (6.5,-.5) -- (6,0);
		\draw[line width=.2 mm] (6.5,-.5) -- (6.5,0);
		\draw[line width=.2 mm] (6.5,-.5) -- (7,0);
		\draw[line width=.2 mm] (7,-.5) -- (6,0);
		\draw[line width=.2 mm] (7,-.5) -- (7,0);
		\draw[line width=.2 mm] (7,-.5) -- (7.5,0);
		\node (A4) at (7,-1)  {$\bf{D_3}$};
		\fill[black!100!](9,0) circle (.07);
		\fill[black!100!](9.5,0) circle (.07);
		\fill[black!100!](10,0) circle (.07);
		\fill[black!100!](10.5,0) circle (.07);
		\fill[black!100!](11,0) circle (.07);
		\fill[black!100!](9.5,-.5) circle (.07);
		\fill[black!100!](10.5,-.5) circle (.07);
		\draw[line width=.2 mm] (9.5,-.5) -- (10.5,-.5);
		\draw[line width=.2 mm] (9,0) -- (9.5,0);
		\draw[line width=.2 mm] (9.5,0) -- (10,0);
		\draw[line width=.2 mm] (10,0) -- (10.5,0);
		\draw[line width=.2 mm] (10.5,0) -- (11,0);
		\draw[line width=.2 mm] (9.5,-.5) -- (9,0);
		\draw[line width=.2 mm] (9.5,-.5) -- (9.5,0);
		\draw[line width=.2 mm] (9.5,-.5) -- (10,0);
		\draw[line width=.2 mm] (10.5,-.5) -- (9.5,0);
		\draw[line width=.2 mm] (10.5,-.5) -- (10,0);
		\draw[line width=.2 mm] (10.5,-.5) -- (11,0);
		\node (A4) at (10,-1)  {$\bf{D_4}$};
		\fill[black!100!](12,0) circle (.07);
		\fill[black!100!](12.5,0) circle (.07);
		\fill[black!100!](13,0) circle (.07);
		\fill[black!100!](13.5,0) circle (.07);
		\fill[black!100!](14,0) circle (.07);
		\fill[black!100!](12.2,-.5) circle (.07);
		\fill[black!100!](13,-.5) circle (.07);
		\draw[line width=.2 mm] (12.2,-.5) -- (13,-.5);
		\draw[line width=.2 mm] (12,0) -- (12.5,0);
		\draw[line width=.2 mm] (12.5,0) -- (13,0);
		\draw[line width=.2 mm] (13,0) -- (13.5,0);
		\draw[line width=.2 mm] (13.5,0) -- (14,0);
		\draw[line width=.2 mm] (12.2,-.5) -- (12,0);
		\draw[line width=.2 mm] (12.2,-.5) -- (13,0);
		\draw[line width=.2 mm] (12.2,-.5) -- (13.5,0);
		\draw[line width=.2 mm] (13,-.5) -- (12,0);
		\draw[line width=.2 mm] (13,-.5) -- (12.5,0);
		\draw[line width=.2 mm] (13,-.5) -- (13,0);
		\draw[line width=.2 mm] (13,-.5) -- (13.5,0);
		\node (A4) at (12.5,-1)  {$\bf{D_5}$};
		\fill[black!100!](0,-1.5) circle (.07);
		\fill[black!100!](.5,-1.5) circle (.07);
		\fill[black!100!](1,-1.5) circle (.07);
		\fill[black!100!](1.5,-1.5) circle (.07);
		\fill[black!100!](2,-1.5) circle (.07);
		\fill[black!100!](1,-2) circle (.07);
		\fill[black!100!](1.5,-2) circle (.07);
		\draw[line width=.2 mm] (0,-1.5) -- (.5,-1.5);
		\draw[line width=.2 mm] (.5,-1.5) -- (1.5,-1.5);
		\draw[line width=.2 mm] (2,-1.5) -- (1.5,-1.5);
		\draw[line width=.2 mm] (1,-2) -- (0,-1.5);
		\draw[line width=.2 mm] (1,-2) -- (.5,-1.5);
		\draw[line width=.2 mm] (1,-2) -- (1,-1.5);
		\draw[line width=.2 mm] (1,-2) -- (1.5,-1.5);
		\draw[line width=.2 mm] (1.5,-2) -- (.5,-1.5);
		\draw[line width=.2 mm] (1.5,-2) -- (1,-1.5);
		\draw[line width=.2 mm] (1.5,-2) -- (1.5,-1.5);
		\draw[line width=.2 mm] (1.5,-2) -- (2,-1.5);
		\node (A4) at (1,-2.5)  {$\bf{D_6}$};
		\fill[black!100!](3,-1.5) circle (.07);
		\fill[black!100!](3.5,-1.5) circle (.07);
		\fill[black!100!](4,-1.5) circle (.07);
		\fill[black!100!](4.5,-1.5) circle (.07);
		\fill[black!100!](5,-1.5) circle (.07);
		\fill[black!100!](3,-2) circle (.07);
		\fill[black!100!](3.5,-2) circle (.07);
		\fill[black!100!](4,-2) circle (.07);
		\draw[line width=.2 mm] (3,-1.5) -- (3.5,-1.5);
		\draw[line width=.2 mm] (3.5,-1.5) -- (4,-1.5);
		\draw[line width=.2 mm] (4,-1.5) -- (4.5,-1.5);
		\draw[line width=.2 mm] (4.5,-1.5) -- (5,-1.5);
		\draw[line width=.2 mm] (3,-2) -- (3,-1.5);
		\draw[line width=.2 mm] (3,-2) -- (3.5,-1.5);
		\draw[line width=.2 mm] (3,-2) -- (4,-1.5);
		\draw[line width=.2 mm] (3.5,-2) -- (3,-1.5);
		\draw[line width=.2 mm] (3.5,-2) -- (4,-1.5);
		\draw[line width=.2 mm] (3.5,-2) --(4.5,-1.5);
		\draw[line width=.2 mm] (4,-2) -- (3,-1.5);
		\draw[line width=.2 mm] (4,-2) -- (3.5,-1.5);
		\draw[line width=.2 mm] (4,-2) -- (4,-1.5);
		\draw[line width=.2 mm] (4,-2) -- (4.5,-1.5);
		\node (A4) at (3.5,-2.5)  {$\bf{D_7}$};
		\fill[black!100!](6,-1.5) circle (.07);
		\fill[black!100!](6.5,-1.5) circle (.07);
		\fill[black!100!](7,-1.5) circle (.07);
		\fill[black!100!](7.5,-1.5) circle (.07);
		\fill[black!100!](6.3,-2) circle (.07);
		%
		%
		\draw[line width=.2 mm] (6,-1.5) -- (6.5,-1.5);
		\draw[line width=.2 mm] (6.5,-1.5) -- (7,-1.5);
		\draw[line width=.2 mm] (7,-1.5) -- (7.5,-1.5);
		\draw[line width=.2 mm] (6.3,-2) -- (6,-1.5);
		\draw[line width=.2 mm] (6.3,-2) -- (6.5,-1.5);
		\node (A4) at (6.5,-2.5)  {$\bf{D_8}$};
		\fill[black!100!](8,-1.5) circle (.07);
		\fill[black!100!](8.5,-1.5) circle (.07);
	\fill[black!100!](9,-1.5) circle (.07);
		\fill[black!100!](9.5,-1.5) circle (.07);
		\fill[black!100!](8.5,-2) circle (.07);
		%
		%
		\draw[line width=.2 mm] (8,-1.5) -- (8.5,-1.5);
		\draw[line width=.2 mm] (8.5,-1.5) -- (9,-1.5);
		\draw[line width=.2 mm] (9,-1.5) -- (9.5,-1.5);
		\draw[line width=.2 mm] (8.5,-2) -- (8,-1.5);
		\draw[line width=.2 mm] (8.5,-2) -- (8.5,-1.5);
		\draw[line width=.2 mm] (8.5,-2) -- (9,-1.5);
		\node (A4) at (8.5,-2.5)  {$\bf{D_9}$};
		\fill[black!100!](10,-1.5) circle (.07);
		\fill[black!100!](10.5,-1.5) circle (.07);
		\fill[black!100!](11,-1.5) circle (.07);
		\fill[black!100!](11.5,-1.5) circle (.07);
		\fill[black!100!](10.7,-2) circle (.07);
		%
		%
		\draw[line width=.2 mm] (10,-1.5) -- (10.5,-1.5);
		\draw[line width=.2 mm] (10.5,-1.5) -- (11,-1.5);
		\draw[line width=.2 mm] (11,-1.5) -- (11.5,-1.5);
		\draw[line width=.2 mm] (10.7,-2) -- (10,-1.5);
		\draw[line width=.2 mm] (10.7,-2) -- (10.5,-1.5);
		\draw[line width=.2 mm] (10.7,-2) -- (11,-1.5);
		\draw[line width=.2 mm] (10.7,-2) -- (11.5,-1.5);
		\node (A4) at (10.5,-2.5)  {$\bf{D_{10}}$};
		\fill[black!100!](12,-1.5) circle (.07);
		\fill[black!100!](12.5,-1.5) circle (.07);
		\fill[black!100!](13,-1.5) circle (.07);
		\fill[black!100!](13.5,-1.5) circle (.07);
		\fill[black!100!](12.5,-2) circle (.07);
		\fill[black!100!](13,-2) circle (.07);
		\draw[line width=.2 mm] (12.5,-2) -- (13,-2);
		\draw[line width=.2 mm] (12,-1.5) -- (12.5,-1.5);
		\draw[line width=.2 mm] (12.5,-1.5) -- (13,-1.5);
		\draw[line width=.2 mm] (13,-1.5) -- (13.5,-1.5);
		\draw[line width=.2 mm] (12.5,-2) -- (12,-1.5);
		\draw[line width=.2 mm] (12.5,-2) -- (12.5,-1.5);
		\draw[line width=.2 mm] (13,-2) -- (13,-1.5);
		\draw[line width=.2 mm] (13,-2) -- (13.5,-1.5);
		\node (A4) at (12.5,-2.5)  {$\bf{D_{11}}$};
		\fill[black!100!](0,-3) circle (.07);
		\fill[black!100!](.5,-3) circle (.07);
		\fill[black!100!](1,-3) circle (.07);
		\fill[black!100!](1.5,-3) circle (.07);
		\fill[black!100!](0,-3.5) circle (.07);
		\fill[black!100!](.5,-3.5) circle (.07);
		\draw[line width=.2 mm] (0,-3) -- (.5,-3);
		\draw[line width=.2 mm] (.5,-3) -- (1,-3);
		\draw[line width=.2 mm] (1,-3) -- (1.5,-3);
		\draw[line width=.2 mm] (0,-3.5) -- (0,-3);
		\draw[line width=.2 mm] (0,-3.5) -- (.5,-3);
		\draw[line width=.2 mm] (0,-3.5) -- (.5,-3.5);
		\draw[line width=.2 mm] (.5,-3.5) -- (1.5,-3);
		\draw[line width=.2 mm] (.5,-3.5) -- (.5,-3);
		\draw[line width=.2 mm] (.5,-3.5) -- (0,-3);
		\node (A4) at (.8,-4)  {$\bf{D_{12}}$};
		\fill[black!100!](2,-3) circle (.07);
		\fill[black!100!](2.5,-3) circle (.07);
		\fill[black!100!](3,-3) circle (.07);
		\fill[black!100!](3.5,-3) circle (.07);
		\fill[black!100!](2.5,-3.5) circle (.07);
		\fill[black!100!](3,-3.5) circle (.07);
		\draw[line width=.2 mm] (2,-3) -- (2.5,-3);
		\draw[line width=.2 mm] (2.5,-3) -- (3,-3);
		\draw[line width=.2 mm] (3,-3) -- (3.5,-3);
		\draw[line width=.2 mm] (2.5,-3.5) -- (2.5,-3);
		\draw[line width=.2 mm] (2.5,-3.5) -- (3,-3);
		\draw[line width=.2 mm] (2.5,-3.5) -- (3,-3.5);
		\draw[line width=.2 mm] (3,-3.5) -- (3,-3);
		\draw[line width=.2 mm] (3,-3.5) -- (2.5,-3);
		\draw[line width=.2 mm] (3,-3.5) -- (2,-3);
		\node (A4) at (2.8,-4)  {$\bf{D_{13}}$};
		\fill[black!100!](4,-3) circle (.07);
		\fill[black!100!](4.5,-3) circle (.07);
		\fill[black!100!](5,-3) circle (.07);
		\fill[black!100!](5.5,-3) circle (.07);
		\fill[black!100!](4.5,-3.5) circle (.07);
		\fill[black!100!](5,-3.5) circle (.07);
		\draw[line width=.2 mm] (4,-3) -- (4.5,-3);
		\draw[line width=.2 mm] (4.5,-3) -- (5,-3);
		\draw[line width=.2 mm] (5,-3) -- (5.5,-3);
		\draw[line width=.2 mm] (4.5,-3.5) -- (4,-3);
		\draw[line width=.2 mm] (4.5,-3.5) -- (4.5,-3);
		\draw[line width=.2 mm] (4.5,-3.5) -- (5,-3);
		\draw[line width=.2 mm] (4.5,-3.5) -- (5,-3.5);
		\draw[line width=.2 mm] (5,-3.5) -- (5,-3);
		\draw[line width=.2 mm] (5,-3.5) -- (4,-3);
		\draw[line width=.2 mm] (5,-3.5) -- (5.5,-3);
		\node (A4) at (4.8,-4)  {$\bf{D_{14}}$};
		\fill[black!100!](6,-3) circle (.07);
		\fill[black!100!](6.5,-3) circle (.07);
		\fill[black!100!](7,-3) circle (.07);
		\fill[black!100!](7.5,-3) circle (.07);
		\fill[black!100!](6.5,-3.5) circle (.07);
		\fill[black!100!](7,-3.5) circle (.07);
		\draw[line width=.2 mm] (6,-3) -- (6.5,-3);
		\draw[line width=.2 mm] (6.5,-3) -- (7,-3);
		\draw[line width=.2 mm] (7.5,-3) -- (7,-3);
		\draw[line width=.2 mm] (6.5,-3.5) -- (6,-3);
		\draw[line width=.2 mm] (6.5,-3.5) -- (6.5,-3);
		\draw[line width=.2 mm] (6.5,-3.5) -- (7.5,-3);
		\draw[line width=.2 mm] (6.5,-3.5) -- (7,-3.5);
		\draw[line width=.2 mm] (7,-3.5) -- (6,-3);
		\draw[line width=.2 mm] (7,-3.5) -- (6.5,-3);
		\draw[line width=.2 mm] (7,-3.5) -- (7,-3);
		\draw[line width=.2 mm] (7,-3.5) -- (7.5,-3);
		\node (A4) at (6.8,-4)  {$\bf{D_{15}}$};
		\fill[black!100!](8,-3) circle (.07);
		\fill[black!100!](8.5,-3) circle (.07);
		\fill[black!100!](9,-3) circle (.07);
		\fill[black!100!](9.5,-3) circle (.07);
		\fill[black!100!](8,-3.5) circle (.07);
		\fill[black!100!](9.5,-3.5) circle (.07);
		\fill[black!100!](8.6,-3.3) circle (.07);
		\draw[line width=.2 mm] (8,-3) -- (8.5,-3);
		\draw[line width=.2 mm] (8.5,-3) -- (9,-3);
		\draw[line width=.2 mm] (9.5,-3) -- (9,-3);
		\draw[line width=.2 mm] (8,-3.5) -- (9.5,-3.5);
		\draw[line width=.2 mm] (8,-3.5) -- (8,-3);
		\draw[line width=.2 mm] (8,-3.5) -- (8.5,-3);
		\draw[line width=.2 mm] (8,-3.5) -- (8.6,-3.3);
		\draw[line width=.2 mm] (9.5,-3.5) -- (9.5,-3);
		\draw[line width=.2 mm] (9.5,-3.5) -- (9,-3);
		\draw[line width=.2 mm] (8.6,-3.3) -- (9.5,-3);
		\draw[line width=.2 mm] (8.6,-3.3) -- (8.5,-3);
		\draw[line width=.2 mm] (8.6,-3.3) -- (8,-3);
		\node (A4) at (8.8,-4)  {$\bf{D_{16}}$};
		\fill[black!100!](10.5,-3) circle (.07);
		\fill[black!100!](11,-3) circle (.07);
		\fill[black!100!](11.5,-3) circle (.07);
		\fill[black!100!](12,-3) circle (.07);
		\fill[black!100!](10.5,-3.5) circle (.07);
		\fill[black!100!](11.2,-3.5) circle (.07);
		\fill[black!100!](12,-3.5) circle (.07);
		\draw[line width=.2 mm] (10.5,-3) -- (11,-3);
		\draw[line width=.2 mm] (11,-3) -- (11.5,-3);
		\draw[line width=.2 mm] (12,-3) -- (11.5,-3);
		\draw[line width=.2 mm] (11,-3) -- (10.5,-3.5);
		\draw[line width=.2 mm] (10.5,-3.5) -- (11.2,-3.5);
		\draw[line width=.2 mm] (10.5,-3.5) -- (10.5,-3);
		\draw[line width=.2 mm] (10.5,-3.5) -- (11.5,-3);
		\draw[line width=.2 mm] (10.5,-3.5) -- (11.5,-3);
		\draw[line width=.2 mm] (11.2,-3.5) -- (12,-3.5);
		\draw[line width=.2 mm] (11.2,-3.5) -- (10.5,-3);
		\draw[line width=.2 mm] (11.2,-3.5) -- (11.5,-3);
		\draw[line width=.2 mm] (11.2,-3.5) -- (12,-3);
		\draw[line width=.2 mm] (12,-3.5) -- (12,-3);
		\draw[line width=.2 mm] (12,-3.5) -- (11.5,-3);
		\draw[line width=.2 mm] (12,-3.5) -- (11,-3);
		\draw[line width=.2 mm] (12,-3.5) -- (10.5,-3);
		\node (A4) at (11.2,-4)  {$\bf{D_{17}}$};
		\fill[black!100!](0,-4.5) circle (.07);
		\fill[black!100!](.5,-4.5) circle (.07);
		\fill[black!100!](1,-4.5) circle (.07);
		\fill[black!100!](1.5,-4.5) circle (.07);
		\fill[black!100!](0,-5) circle (.07);
		\fill[black!100!](.5,-5) circle (.07);
		\fill[black!100!](1.5,-5) circle (.07);
		\draw[line width=.2 mm] (0,-4.5) -- (.5,-4.5);
		\draw[line width=.2 mm] (.5,-4.5) -- (1,-4.5);
		\draw[line width=.2 mm] (1,-4.5) -- (1.5,-4.5);
		\draw[line width=.2 mm] (0,-5) -- (0,-4.5);
		\draw[line width=.2 mm] (0,-5) -- (.5,-4.5);
		\draw[line width=.2 mm] (0,-5) -- (1.5,-4.5);
		\draw[line width=.2 mm] (0,-5) -- (.5,-5);
		\draw[line width=.2 mm] (.5,-5) -- (0,-4.5);
		\draw[line width=.2 mm] (.5,-5) -- (.5,-4.5);
		\draw[line width=.2 mm] (.5,-5) -- (1,-4.5);
		\draw[line width=.2 mm] (.5,-5) -- (1.5,-4.5);
		\draw[line width=.2 mm] (.5,-5) -- (1.5,-5);
		\draw[line width=.2 mm] (1.5,-5) -- (0,-4.5);
		\draw[line width=.2 mm] (1.5,-5) -- (1,-4.5);
		\draw[line width=.2 mm] (1.5,-5) -- (1.5,-4.5);
		\node (A4) at (.8,-5.3)  {$\bf{D_{18}}$};
		\fill[black!100!](2,-4.5) circle (.07);
		\fill[black!100!](2.5,-4.5) circle (.07);
		\fill[black!100!](3,-4.5) circle (.07);
		\fill[black!100!](3.5,-4.5) circle (.07);
		\fill[black!100!](2,-5.2) circle (.07);
		\fill[black!100!](3.5,-5.2) circle (.07);
		\fill[black!100!](2.5,-5) circle (.07);
		\fill[black!100!](3,-5) circle (.07);
		\draw[line width=.2 mm] (2,-4.5) -- (2.5,-4.5);
		\draw[line width=.2 mm] (2.5,-4.5) -- (3,-4.5);
		\draw[line width=.2 mm] (3,-4.5) -- (3.5,-4.5);
		\draw[line width=.2 mm] (2,-5.2) -- (2,-4.5);
		\draw[line width=.2 mm] (2,-5.2) -- (3.5,-5.2);
		\draw[line width=.2 mm] (2,-5.2) -- (2.5,-4.5);
		\draw[line width=.2 mm] (2,-5.2) -- (2.5,-5);
		\draw[line width=.2 mm] (3.5,-5.2) -- (3.5,-4.5);
		\draw[line width=.2 mm] (2.5,-5) -- (3.5,-4.5);
		\draw[line width=.2 mm] (3.5,-5.2) -- (3,-4.5);
		\draw[line width=.2 mm] (3.5,-5.2) -- (3,-5);
		\draw[line width=.2 mm] (2.5,-5) -- (2.5,-4.5);
		\draw[line width=.2 mm] (2.5,-5) -- (2,-4.5);
		\draw[line width=.2 mm] (3,-5) -- (3,-4.5);
		\draw[line width=.2 mm] (3,-5) -- (2,-4.5);
		\draw[line width=.2 mm] (3,-5) -- (3.5,-4.5);
		\node (A4) at (2.8,-5.5)  {$\bf{D_{19}}$};
		\fill[black!100!](4,-4.5) circle (.07);
		\fill[black!100!](4.5,-4.5) circle (.07);
		\fill[black!100!](5,-4.5) circle (.07);
		\fill[black!100!](4,-5) circle (.07);
		\fill[black!100!](4.5,-5) circle (.07);
		\draw[line width=.2 mm] (4,-4.5) -- (4.5,-4.5);
		\draw[line width=.2 mm] (4.5,-4.5) -- (5,-4.5);
		\draw[line width=.2 mm] (4,-5) -- (4,-4.5);
		\draw[line width=.2 mm] (4,-5) -- (4.5,-5);
		\draw[line width=.2 mm] (4,-5) -- (4.5,-4.5);
		\draw[line width=.2 mm] (4.5,-5) -- (4.5,-4.5);
		\draw[line width=.2 mm] (4.5,-5) -- (4,-4.5);
		\node (A4) at (4.2,-5.3)  {$\bf{D_{20}}$};
		\fill[black!100!](5.5,-4.5) circle (.07);
		\fill[black!100!](6,-4.5) circle (.07);
		\fill[black!100!](6.5,-4.8) circle (.07);
		\fill[black!100!](5.5,-5.2) circle (.07);
		\fill[black!100!](6,-5.2) circle (.07);
		\draw[line width=.2 mm] (5.5,-4.5) -- (6,-4.5);
		\draw[line width=.2 mm] (5.5,-4.5) -- (5.5,-5.2);
		\draw[line width=.2 mm] (5.5,-4.5) -- (6,-5.2);
		\draw[line width=.2 mm] (6,-4.5) -- (6,-5.2);
		\draw[line width=.2 mm] (6,-4.5) -- (5.5,-5.2);
		\draw[line width=.2 mm] (6.5,-4.8) -- (6,-4.5);
		\draw[line width=.2 mm] (6.5,-4.8) -- (6,-5.2);
		\draw[line width=.2 mm] (5.5,-5.2) -- (6,-5.2);
		\node (A4) at (5.8,-5.5)  {$\bf{D_{21}}$};
		\fill[black!100!](7,-4.5) circle (.07);
		\fill[black!100!](8,-4.5) circle (.07);
		\fill[black!100!](7,-5.2) circle (.07);
		\fill[black!100!](8,-5.2) circle (.07);
		\fill[black!100!](7.5,-4.9) circle (.07);
		\draw[line width=.2 mm] (7,-4.5) -- (7.5,-4.9);
		\draw[line width=.2 mm] (7,-4.5) -- (7,-5.2);
		\draw[line width=.2 mm] (8,-4.5) -- (7.5,-4.9);
		\draw[line width=.2 mm] (8,-4.5) -- (8,-5.2);
		\draw[line width=.2 mm] (7,-5.2) -- (7.5,-4.9);
		\draw[line width=.2 mm] (8,-5.2) -- (7.5,-4.9);
		\node (A4) at (7.5,-5.5)  {$\bf{D_{22}}$};
		\fill[black!100!](9,-4.5) circle (.07);
		\fill[black!100!](8.5,-4.8) circle (.07);
		\fill[black!100!](9.5,-4.8) circle (.07);
		\fill[black!100!](8.7,-5.3) circle (.07);
		\fill[black!100!](9.3,-5.3) circle (.07);
		\draw[line width=.2 mm] (9,-4.5) -- (8.5,-4.8);
		\draw[line width=.2 mm] (9,-4.5) -- (9.5,-4.8);
		\draw[line width=.2 mm] (8.5,-4.8) -- (8.7,-5.3);
		\draw[line width=.2 mm] (8.7,-5.3) -- (9.3,-5.3);
		\draw[line width=.2 mm] (9.3,-5.3) -- (9.5,-4.8);
		\node (A4) at (9,-5.7)  {$\bf{D_{23}}$};
		\fill[black!100!](10.5,-4.5) circle (.07);
		\fill[black!100!](10,-4.8) circle (.07);
		\fill[black!100!](11,-4.8) circle (.07);
		\fill[black!100!](10.2,-5.3) circle (.07);
		\fill[black!100!](10.8,-5.3) circle (.07);
		\draw[line width=.2 mm] (10.5,-4.5) -- (11,-4.8);
		\draw[line width=.2 mm] (10.5,-4.5) -- (11,-4.8);
		\draw[line width=.2 mm] (10,-4.8) -- (10.2,-5.3);
		\draw[line width=.2 mm] (10.2,-5.3) -- (10.8,-5.3);
		\draw[line width=.2 mm] (10,-4.8) -- (10.5,-4.5);
		\draw[line width=.2 mm] (10.8,-5.3) -- (11,-4.8);
		\draw[line width=.2 mm] (10.5,-4.5) -- (10.8,-5.3);
		\draw[line width=.2 mm] (10.5,-4.5) -- (10.2,-5.3);
		\draw[line width=.2 mm] (10,-4.8) -- (10.8,-5.3);
		\draw[line width=.2 mm] (10,-4.8) -- (11,-4.8);
		\draw[line width=.2 mm] (10.2,-5.3) -- (11,-4.8);
		\node (A4) at (10.5,-5.7)  {$\bf{D_{24}}$};
		\end{tikzpicture}
		\caption{Reduced graphs of rank $5$}
		\label{fig:rank 5 graphs}	
	\end{figure}
	Now, we state the main theorem in \cite{LAA-rank-5-graph-charecterization} that characterized all graphs of rank $5.$
	\begin{lemma}[\cite{LAA-rank-5-graph-charecterization}]\label{lemma: cha of graphs of rank 5}
	$r(G)=5$ if and only if $\Gamma\in\mathcal{M}(D_1, D_2, \cdots, D_{24}),$ where the graphs $D_1, D_2, \cdots, D_{24}$ are described in Figure \ref{fig:rank 5 graphs}.
	\end{lemma}

\begin{proposition}\label{prop:rank-5-d-td=2}
Let $g(G)=3$ and $r(G)=5$. Then $\gamma(G)=\gamma_t(G)=2$ if and only if $G$ belongs to one of the following classes:
\begin{enumerate}
\item $G\in\mathcal{M}\bigl(D_2,D_5,D_6,D_8,D_9,D_{11},D_{12},D_{13},D_{14},D_{16},D_{17},D_{19}\bigr)$; or
\item $G\in\mathcal{M}\bigl(D_{10},D_{15},D_{18},D_{20},D_{21},D_{22},D_{24}\bigr)$, with $m_i\geq 2$ for at least one index $i$ corresponding to a dominating vertex $v_i$ of the base graph.
\end{enumerate}
\end{proposition}

\begin{proof}
We first prove the sufficiency. Suppose that
\[G\in\mathcal{M}\bigl(D_2,D_5,D_6,D_8,D_9,D_{11},D_{12},D_{13},D_{14},D_{16},D_{17},D_{19}\bigr).\]
For each
$j\in\{2,5,6,8,9,11,12,13,14,16,17,19\}$,
it can be easily check that $\gamma(D_j)=\gamma_t(D_j)=2$. Hence, by Lemmas~\ref{Lemma:TD(GammaI)EqualToTD(GammaOm)} and~\ref{lemma:dn=tdn-implydn=dn-inmultiplication}, it follows that
$\gamma(G)=\gamma_t(G)=2$ for every vector $m$ of positive integers.

Now suppose that
$G\in\mathcal{M}\bigl(D_{10},D_{15},D_{18},D_{20},D_{21},D_{22},D_{24}\bigr)$,
where $m_i\geq2$ for at least one $i$, and let $v_i$ denote the corresponding dominating vertex. Since each of the graphs $D_{10},D_{15},D_{18},D_{20},D_{21},D_{22},D_{24}$ has a dominating vertex, Lemma~\ref{lemma:dn=1-tdn=dn=2-iff} gives
$\gamma(G)=\gamma_t(G)=2$.

For the converse, let $G$ be a graph with $g(G)=3$ and $r(G)=5$ such that $\gamma(G)=\gamma_t(G)=2$. By Lemma~\ref{lemma:girth-blowup} and Lemma~\ref{lemma: cha of graphs of rank 4}, we have
$G\in\mathcal{M}\bigl(D_1,\ldots,D_{22},D_{24}\bigr)$.

The graphs $D_1,D_3,D_4,D_7$ can be excluded. Indeed,
$\gamma_t(D_i)=3$ for each $i\in\{1,3,4,7\}$. Consequently, by Lemma~\ref{Lemma;Dom(GOm)geDom(G)}, we have
$\gamma_t(D_i\bigodot m)\geq3$
for every vector $m$ of positive integers. Thus, none of these graphs can occur as a component in a graph satisfying $\gamma_t(G)=2$.

It remains to consider the classes
$\mathcal{M}\bigl(D_{10},D_{15},D_{18},D_{20},D_{21},D_{22},D_{24}\bigr)$.
If $m_i=1$ for every corresponding dominating vertex $v_i$, then the resulting graph has a dominating vertex and hence $\gamma(G)=1$, contradicting the assumption that $\gamma(G)=2$. Therefore, $m_i\geq2$ for at least one $i$.

Hence, $G$ belongs to one of the two classes stated in the proposition, and the proof is complete.
\end{proof}

\begin{theorem}\label{thm:girth-3-rank-2-5}
Let $G$ be a graph with $g(G)=3$ and $r(G)\in\{2,3,4,5\}$. Then $\gamma(G)=\gamma_t(G)=2$ if and only if $G$ belongs to one of the following classes:
\begin{enumerate}
\item[\textup{(i)}] $G\in\mathcal{M}(K_3)$, with $m_i\geq 2$ for at least one $i\in[3]$;

\item[\textup{(ii)}] $G\in\mathcal{M}\bigl(B_2,B_4,B_5\bigr)$;

\item[\textup{(iii)}] $G\in\mathcal{M}\bigl(B_1,K_4\bigr)$, with $m_i\geq 2$ for at least one index $i$ corresponding to a dominating vertex $v_i$ of the base graph;

\item[\textup{(iv)}] $G\in\mathcal{M}\bigl(D_2,D_5,D_6,D_8,D_9,D_{11},D_{12},D_{13},D_{14},D_{16},D_{17},D_{19}\bigr)$;

\item[\textup{(v)}] $G\in\mathcal{M}\bigl(D_{10},D_{15},D_{18},D_{20},D_{21},D_{22},D_{24}\bigr)$, with $m_i\geq 2$ for at least one index $i$ corresponding to a dominating vertex $v_i$ of the base graph.
\end{enumerate}

\end{theorem}

\begin{proof}
Immediate from Propositions~\ref{prop:rank-2,3-d-td=2}, \ref{prop:rank-4-d-td=2}, and \ref{prop:rank-5-d-td=2}.
\end{proof}

\section{Two Extremes of Diameter-Two Graphs}

Diameter-two graphs form a natural testing ground for the equality $\gamma(G)=\gamma_t(G)$: this is precisely the setting in which Goddard and Henning \cite{GoddardHenning2002} established the sharp bound $\gamma_t(G)\le 3$ for planar graphs, leaving open exactly which diameter-two graphs attain $\gamma_t(G)=2$ rather than $3$. We answer this at two structural extremes. At the dense end, a diameter-two graph with a pendant vertex satisfies $\gamma_t(G)=2$ automatically whenever $\gamma(G)=2$. At the sparse end, we consider the extremal family $\mathcal{H}$ of diameter-two graphs with no dominating vertex and minimum size, determined by Erd\H{o}s and R\'enyi \cite{ERDOS-RENYEL} and classified by Henning and Southey \cite{Henning-Michael}; here every member satisfies $\gamma_t(G)\in\{3,6\}$, so $\gamma=\gamma_t=2$ never occurs. A pendant vertex forces total-domination equality; extremal sparsity forbids it outright.

\subsection{Graphs with a pendant vertex}

\begin{lemma}\label{Lemma:pendant-not-in-dom}
Let $\operatorname{diam}(G)=2$ and $\gamma(G)\ge 2$. Then no minimum dominating set of $G$ contains a pendant vertex.
\end{lemma}

\begin{proof}
Suppose, for a contradiction, that 
$D=\{v,v_1,\dots,v_{k-1}\}$
is a minimum dominating set of $G$, where $\gamma(G)=k$ and $\deg(v)=1$. Let $x$ be the unique neighbor of $v$.
Since $D$ is a minimum dominating set, every vertex of $D$ must dominate at least one vertex outside $D$; otherwise, it could be removed while preserving domination, contradicting the minimality of $D$. In particular, $v$ must dominate a vertex outside $D$. As $\deg(v)=1$, its only neighbor is $x$, so $x \notin D$. Moreover, $v$ is not adjacent to any of the vertices $v_1,\dots,v_{k-1}$.

We claim that $x$ is not adjacent to any vertex $v_i$ for $i\in\{1,\dots,k-1\}$. Indeed, if $x$ were adjacent to some $v_i$, then $v$ could be removed from $D$ and the set
$D'=\{v_1,\dots,v_{k-1}\}$
would still dominate $G$, since $v$ would then be dominated by $v_i$ through $x$, and all other vertices would remain dominated. This contradicts the minimality of $D$. Hence, $x$ is not adjacent to any $v_i$.
Since $G$ is connected and $x \notin D$, the vertex $x$ must be adjacent to some vertex $y$ that is dominated by some $v_i$ with $i\in\{1,\dots,k-1\}$. Therefore, there exists a path
$v \sim x \sim y \sim v_i$
of length $3$ between $v$ and $v_i$. As $v$ is adjacent only to $x$, and $x$ is not adjacent to any $v_i$, no shorter path between $v$ and $v_i$ exists. Thus,
$d(v,v_i)=3.$
It follows that $\operatorname{diam}(G)\ge 3$, contradicting the assumption that $\operatorname{diam}(G)=2$.
Therefore, no minimum dominating set of $G$ contains a pendant vertex.
\end{proof}
  
\begin{theorem}
Let $G$ be a graph with a pendant vertex, $\operatorname{diam}(G)=2$, and $\gamma(G)=2$. Then $\gamma_t(G)=2$.
\end{theorem}

\begin{proof}
Let $a$ be a pendant vertex of $G$ with $\deg(a)=1$, and let $u$ be its unique neighbor. By Lemma~\ref{Lemma:pendant-not-in-dom}, no minimum dominating set of $G$ contains a pendant vertex. Since $\gamma(G)=2$, it follows that every minimum dominating set must contain $u$.
Let $D=\{u,v\}$ be a minimum dominating set of $G$. If $u\sim v$, then each vertex of $D$ is adjacent to another vertex of $D$, and hence $D$ is a total dominating set. Therefore, $\gamma_t(G)=2$.

Suppose, on the contrary, that $u$ and $v$ are not adjacent. Since $a$ is adjacent only to $u$, any path from $a$ to $v$ must pass through $u$. Because $u \nsim v$, there must exist a vertex $y$ such that
$u \sim y \sim v.$
Consequently,
$a \sim u \sim y \sim v$
is a path of length $3$ from $a$ to $v$. 
Moreover, no shorter path can exist: $a$ is adjacent only to $u$, and $u$ is not adjacent to $v$. Hence,
$d(a,v)=3,$
which implies that $\operatorname{diam}(G)\ge 3$. This contradicts the assumption that $\operatorname{diam}(G)=2$.
Therefore, $u$ must be adjacent to $v$, and hence $D$ is a total dominating set of $G$. Consequently, $\gamma_t(G)=2$.
\end{proof}

\subsection{The extremal family $\mathcal{H}$ of Erdős-Rényi and Henning-Southey}

Let $\mathcal{H}$ be the family of graphs that: 
\begin{enumerate}
\item contains $C_5, G_0$ (as in Figure \ref{fig:graph of diam 2;cycle 0f length $5$ and its duplication}) and the Petersen graph; and 
\item is closed under degree-$2$ vertex duplication,
\end{enumerate}
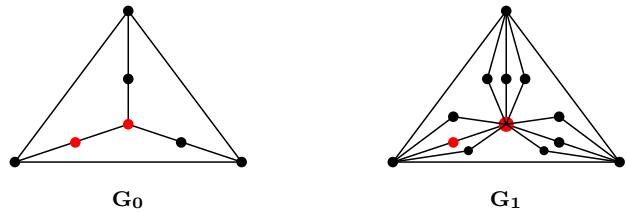
\begin{figure}[H]
\tiny
\tikzstyle{ver}=[H]
			\tikzstyle{vert}=[circle, draw, fill=black!100, inner sep=0pt, minimum width=4pt]
			\tikzstyle{vertex}=[circle, draw, fill=black!00, inner sep=0pt, minimum width=4pt]
			\tikzstyle{edge} = [draw,thick,-]
			\tikzstyle{node_style} = [circle,draw=blue,fill=blue!20!,font=\sffamily\Large\bfseries]
			\centering
			\begin{tikzpicture}[scale=1]
			\tikzstyle{edge_style} = [draw=black, line width=2mm, ]
			\tikzstyle{node_style} = [draw=blue,fill=blue!00!,font=\sffamily\Large\bfseries]
			\node (A4) at (1.5,-3.5)  {$\bf{G_0}$};
			\node (A4) at (6.5,-3.5)  {$\bf{G_1}$};
			\draw[line width=.2 mm] (0,-3) -- (3,-3);
			\draw[line width=.2 mm] (0,-3) -- (1.5,-1);
			\draw[line width=.2 mm] (3,-3) -- (1.5,-1);
			\fill[black!100!](0,-3) circle (.07);
			\fill[black!100!](3,-3) circle (.07);
			\fill[black!100!](1.5,-1) circle (.07);
			\draw[line width=.2 mm] (0,-3) -- (1.5,-2.5);
			\draw[line width=.2 mm] (1.5,-1) -- (1.5,-2.5);
			\draw[line width=.2 mm] (3,-3) -- (1.5,-2.5);
			\fill[red!100!](1.5,-2.5) circle (.07);
			\fill[red!100!](.8,-2.74) circle (.07);
			\fill[black!100!](2.2,-2.74) circle (.07);
			\fill[black!100!](1.5,-1.9) circle (.07);
			%
			\draw[line width=.2 mm] (5,-3) -- (8,-3);
			\draw[line width=.2 mm] (5,-3) -- (6.5,-1);
			\draw[line width=.2 mm] (8,-3) -- (6.5,-1);
			\fill[black!100!](5,-3) circle (.07);
			\fill[black!100!](8,-3) circle (.07);
			\fill[black!100!](6.5,-1) circle (.07);
			\draw[line width=.2 mm] (5,-3) -- (6.5,-2.5);
			\draw[line width=.2 mm] (6.5,-1) -- (6.5,-2.5);
			\draw[line width=.2 mm] (8,-3) -- (6.5,-2.5);
			\fill[red!100!](6.5,-2.5) circle (.1);
			\fill[red!100!](5.8,-2.74) circle (.07);
			\fill[black!100!](7.2,-2.74) circle (.07);
			\fill[black!100!](6.5,-1.9) circle (.07);
			\fill[black!100!](6.25,-1.9) circle (.07);
			\fill[black!100!](6.75,-1.9) circle (.07);
			\draw[line width=.2 mm] (6.25,-1.9) -- (6.5,-1);
			\draw[line width=.2 mm] (6.25,-1.9) -- (6.5,-2.5);
			\draw[line width=.2 mm] (6.75,-1.9) -- (6.5,-1);
			\draw[line width=.2 mm] (6.75,-1.9) -- (6.5,-2.5);
			\fill[black!100!](5.8,-2.4) circle (.07);
			\draw[line width=.2 mm] (6.5,-2.5) -- (5.8,-2.4);
			\draw[line width=.2 mm] (5,-3) -- (5.8,-2.4);
			\fill[black!100!](6,-2.85) circle (.06);
			\draw[line width=.2 mm] (6.5,-2.5) -- (6,-2.85);
			\draw[line width=.2 mm] (5,-3) -- (6,-2.85);
			\fill[black!100!](7.2,-2.4) circle (.07);
			\draw[line width=.2 mm] (6.5,-2.5) -- (7.2,-2.4);
			\draw[line width=.2 mm] (8,-3) -- (7.2,-2.4);
			\fill[black!100!](7,-2.85) circle (.06);
			\draw[line width=.2 mm] (6.5,-2.5) -- (7,-2.85);
			\draw[line width=.2 mm] (8,-3) -- (7,-2.85);

\end{tikzpicture}
\caption{The graph $G_1$ is obtained from $G_0$ by duplicating six degree-$2$ vertices.}
\label{fig:graph of diam 2;cycle 0f length $5$ and its duplication}	
\end{figure}	
In \cite{ERDOS-RENYEL}, Erd\H{o}s and R\'{e}nyi, proved the following classical result.
\begin{theorem}[\cite{ERDOS-RENYEL}]\label{Thm: s geq 2n-5, by ERDOS-RENYEL}
If $G$ is a diameter-$2$ graph of order $n$ and size $s$ with no dominating vertex, then $s\geq 2n-5.$
\end{theorem}
After that, in \cite{Henning-Michael}, Henning and Southey classified all graphs $G$ such that $s=2n-5.$ In fact, they proved the following:
\begin{theorem}[\cite{Henning-Michael}]\label{thm H-M: dim2 m  geq 2n-5}
If $G$ is a diameter-$2$ graph of order $n$ and size $s$ with no dominating vertex, then $s\geq 2n-5$ with equality if and only if $G\in\mathcal{H}.$ 
\end{theorem}
\begin{lemma}\label{Lemma:DomiOfGamma0Is2ButTDIs3}
Let $G_0$ be the graph depicted in Figure~\ref{fig:graph of diam 2;cycle 0f length $5$ and its duplication}. Then $\gamma(G_0)=2$ and $\gamma_t(G_0)=3$.     
\end{lemma}

\begin{proof}
It is easy to see that the set $D=\{v,u_1\}$ is a dominating set, where $\deg(v)=\deg(u_1)=3$, $v$ is the center vertex (shown in red), and $u_1$ is an outside vertex (shown in black). Hence $\gamma(G_0)\le 2$. Since no single vertex dominates all vertices of $G_0$, we obtain $\gamma(G_0)=2$.

Now we claim that any minimum dominating set $D_1$ must contain the vertex $v$ and one of the vertices $u_1,u_2,u_3$. Suppose $v\notin D_1$. Then $v_i\in D_1$ for some $i\in[3]$. Without loss of generality, assume that $v_1\in D_1$. Observe that $u_i$ is not adjacent to $v_1$ for each $i\in[3]\setminus\{2\}$. Hence there must exist a vertex $x\in V(G_0)$ such that both $u_1$ and $u_3$ are adjacent to $x$ in order for the set $\{v_1,x\}$ to dominate $G_0$. Clearly $x\in\{u_1,u_2,u_3\}$.

If $x=u_1$, then the vertex $v_2$ is not dominated by the set $\{v_1,u_1\}$, contradicting $\gamma(G_0)=2$. The other possibilities lead to similar contradictions. Therefore $v\in D_1$.

By a similar argument, $D_1$ must also contain one of the vertices $u_1,u_2,u_3$. Hence $D_1=\{v,u_i\}$ for some $i\in[3]$. Consequently, any minimum dominating set of $G_0$ contains the vertex $v$, which has no neighbor in $D_1$. Thus no dominating set of size $2$ can be a total dominating set, and therefore $\gamma_t(G_0)\ge 3$.

Finally, the set $\{v,v_3,u_1\}$ forms a total dominating set of $G_0$. Hence $\gamma_t(G_0)=3$.
\end{proof}

\begin{theorem}\label{Thm:TotalDomOfGamma3;WhenGammaIsaMemberOfH}
Let $G\in \mathcal{H}.$ Then $\gamma_t(G)$ is either $3$ or $6.$    
\end{theorem}
\begin{proof}
It is well known that $\gamma_t(C_5)=3$ and the total domination number of the Petersen graph is $6.$ By Lemma \ref{Lemma:DomiOfGamma0Is2ButTDIs3}, $\gamma_t(G_0)=3.$ Now, let $\Gamma\in \mathcal{H},$ then by Lemma \ref{Lemma:TD(GammaI)EqualToTD(GammaOm)}, $\gamma_t(G)=3.$  
\end{proof}

\section*{Acknowledgements} 
The author gratefully acknowledges financial support from the NBHM research project (Reference No.~02011/29/2025NBHM(RP)/R\&D~II/11951), and sincerely thanks the National Board for Higher Mathematics (NBHM), India, for funding this research. The author also gratefully acknowledges the excellent research environment and facilities provided by Dhirubhai Ambani University, Gandhinagar, Gujarat. 

\subsection*{Data Availability Statements}
Data sharing not applicable to this article as no datasets were generated or analysed during the current study.

\subsection*{Competing Interests} The authors have no competing interests to declare that are relevant to the content of this article.

\bibliographystyle{amsplain}
\bibliography{gen-inv-lcp(sb).bib}
\end{document}